\documentclass[11pt]{article}

\newcommand{\tn}{\textnormal}
\newcommand{\mb}{\mathbb} \newcommand{\lp}{\left(}

\newcommand{\rp}{\right)} \newcommand{\lb}{\left\lbrace}
\newcommand{\rb}{\right\rbrace} 
	
\newcommand{\mc}{\mathcal}

\usepackage {epsfig,amsmath,euscript}

\usepackage[latin1]{inputenc}

\usepackage[english]{babel}
\usepackage{color}
\usepackage{subcaption}
\usepackage{amsmath}
\usepackage[normalem]{ulem}
\usepackage{amsthm}
\usepackage{amssymb}

\usepackage{ae}
\usepackage{float}
\usepackage[T1]{fontenc}
\usepackage{graphicx}
\usepackage{epstopdf}
\usepackage{longtable}
\usepackage{comment}
\usepackage{color}
\definecolor{rred}{rgb}{0.7,0.0,0.2}
\definecolor{bblue}{rgb}{0.2,0.0,0.7}

\newcommand{\secref}[1]{Section \ref{sec:#1}}

\newtheorem{theorem}{Theorem}
\newtheorem{lemma}{Lemma}

\newcommand{\seclab}[1]{\label{sec:#1}}

\newcommand{\eqlab}[1]{\label{eq:#1}}
\renewcommand{\eqref}[1]{(\ref{eq:#1})}

\newcommand{\figref}[1]{Fig.~\ref{fig:#1}}
\newcommand{\figlab}[1]{\label{fig:#1}}

\newcommand{\lemmaref}[1]{Lemma~\ref{lemma:#1}}
\newcommand{\lemmalab}[1]{\label{lemma:#1}}

\newcommand{\thmref}[1]{Theorem~\ref{theorem:#1}}
\newcommand{\thmlab}[1]{\label{theorem:#1}}

\newcommand{\asuref}[1]{Assumption~\ref{assumption:#1}}
\newcommand{\asulab}[1]{\label{assumption:#1}}

\newtheorem{cor}{Corollary}
\newtheorem{asu}{Assumption}

\makeatletter
\newcommand{\tpitchfork}{%
	\vbox{
		\baselineskip\z@skip
		\lineskip-.52ex
		\lineskiplimit\maxdimen
		\m@th
		\ialign{##\crcr\hidewidth\smash{$-$}\hidewidth\crcr$\pitchfork$\crcr}
	}%
}
\makeatother
\usepackage{tikz}

\usetikzlibrary{positioning}
\usetikzlibrary{arrows}
\usetikzlibrary{arrows.meta}
\usetikzlibrary{calc}

\title{Entry-exit functions in fast-slow systems with intersecting eigenvalues}
\author{Panagiotis Kaklamanos{$^\dagger$}, Christian Kuehn{$^\star$}, Nikola Popovi{\'c}{$^\dagger$}, Mattia Sensi{$^\ast$}\\[1em]
\small $^\dagger$Maxwell Institute for Mathematical Sciences and School of Mathematics,
University of Edinburgh,\\\small James Clerk Maxwell Building, King's Buildings, Peter Guthrie Tait Road,\\\small 
Edinburgh, EH9 3FD,
United Kingdom\\
\small $^\star$Technical University of Munich, Department of Mathematics, Boltzmannstrasse 3,\\\small 85748 Garching bei M\"unchen, Germany\\
\small $^\ast$MathNeuro Team, Inria at Universit\' e C\^ote d'Azur, 2004 Rte des Lucioles, 06410 Biot, France\\
\small Corresponding author. Email: \texttt{mattia.sensi@inria.fr}}
\date{\today}

\begin{document}

\maketitle

\begin{abstract}
We study delayed loss of stability in a class of fast-slow systems with two fast variables and one slow one, where the linearisation of the fast vector field along a one-dimensional critical manifold has two real eigenvalues which intersect before the accumulated contraction and expansion are balanced along any individual eigendirection.
That interplay between eigenvalues and eigendirections renders the use of known entry-exit relations unsuitable for calculating the point at which trajectories exit neighbourhoods of the given manifold. We illustrate the various qualitative scenarios that are possible in the class of systems considered here, and we propose novel formulae for the entry-exit functions that underlie the phenomenon of delayed loss of stability therein.
\end{abstract}

\section{Introduction}
\seclab{intro}
The phenomenon of delayed loss of stability in two-dimensional fast-slow systems of the form
\begin{subequations}\eqlab{2d}
\begin{align}
    x' &= \varepsilon, \eqlab{2d-a}\\
    z' &=Z(x,z;\varepsilon), \eqlab{2d-b}
\end{align}
\end{subequations}
with $\varepsilon>0$ sufficiently small and $Z:\mb{R}^3\to\mb{R}$ smooth, has been extensively studied \cite{de2008smoothness,de2016entry,jardon2021geometric,liu2000exchange,neishtadt1987persistence,neishtadt1988persistence,schecter1985persistent,schecter2008exchange}. In particular, we assume here that the $x$-axis is invariant under the flow of \eqref{2d}, i.e., that $Z(x,0;\varepsilon)=0$, and that it undergoes a change of stability at $x=0$: specifically, we take the $x$-axis to be attracting and repelling for $x<0$ and $x>0$, respectively, with $\partial_zZ(x,0;\varepsilon)<0$ for $x<0$ and $\partial_zZ(x,0;\varepsilon)>0$ for $x>0$, respectively. Delayed loss of stability can then be characterised as follows: trajectories of \eqref{2d} that enter a $ \delta $-neighbourhood of the $x$-axis at a point with $x=x_0<0$ and $\delta$ sufficiently small evolve close thereto until the accumulated contraction is balanced by accumulated expansion instead of diverging immediately from the $x$-axis after crossing $x=0$; cf. \figref{2dsf}. Contraction and expansion are balanced at a point with $x=x_1+o(1)$, where $x_1$ is obtained by solving
\begin{align}
\int_{x_0}^{x_1} \frac{\partial Z}{\partial z}(x,0;0)\tn{d}x =0. \eqlab{wio-pl1}
\end{align}
Equation \eqref{wio-pl1} is known as the \textit{entry-exit relation}; correspondingly, the left-hand side therein is the \textit{way-in/way-out} or \textit{entry-exit} function, see  \cite{de2008smoothness,de2016entry,jardon2021geometric,liu2000exchange,neishtadt1987persistence,neishtadt1988persistence,schecter1985persistent,schecter2008exchange} for details.

\begin{figure}[H]\centering
	\begin{subfigure}[b]{0.3\textwidth}
		\centering
		\begin{tikzpicture}
		\node at (0,0){
			\includegraphics[width=0.9\textwidth]{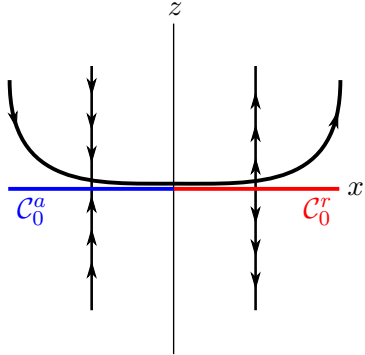}};
		\node at (2.4,0.0) {$x$};
		\node at (0.0,2.4) {$z$};
		
		\node at (-1.9,-0.3) {$\textcolor{blue}{\mc{C}_0^a}$};
		\node at (1.9,-0.3) {$\textcolor{red}{\mc{C}_0^r}$};
		\end{tikzpicture}
	\end{subfigure}
	\caption{Invariant manifold ($ x $-axis) with change of stability at the origin and a canard trajectory that undergoes delayed loss of stability; see also \cite{de2008smoothness}. (Here and in the following, attracting portions of a critical manifold are indicated in blue, while repelling portions are shown in red.)}
	\figlab{2dsf}
\end{figure}

In the language of geometric singular perturbation theory (GSPT) \cite{fenichel1979geometric}, equation~\eqref{2d} has a one-dimensional critical manifold
\begin{align*}
    \mc{C}_0 = \lb (x,z)\in \mb{R}^2~\lvert~ Z(x,z;0)=0\rb
\end{align*}
along which the stability changes from attracting to repelling. A subset of $\mc{C}_0$ is called \textit{normally hyperbolic} if $\partial_zZ(x,z;0)\neq0$. A normally hyperbolic portion of $\mc{C}_0$ is \textit{attracting} if $\partial_zZ(x,z;0)<0$, and is denoted by $\mc{C}_0^a$; correspondingly, it is \textit{repelling} if $\partial_zZ(x,z;0)>0$, and denoted by $\mc{C}_0^r$. Therefore, for \eqref{2d}, we write
	\begin{align*}
	\mc{C}_0^a = \lb (x,z)\in \mc{C}_0~\lvert~x<0\rb \quad\text{and}\quad
	\mc{C}_0^r=\lb (x,z)\in \mc{C}_0~\lvert~x>0\rb.
	\end{align*} 
	The origin is then a \textit{non-hyperbolic} point, since  it holds that $\partial_zZ(0,0;0)=0$.

Trajectories of \eqref{2d} with $\varepsilon>0$ sufficiently small which, after crossing a neighbourhood of a non-hyperbolic point, evolve close to a repelling manifold for a considerable amount of time, are called \textit{canard trajectories} \cite{dumortier1996canard,krupa2001extending,szmolyan2001canards}. Trajectories that experience delayed loss of stability along an invariant manifold of \eqref{2d}, as outlined above, are therefore canard trajectories. However, we emphasise that the above merely represents one example of a canard, and that a plethora of delicate canard phenomena can occur in other planar fast-slow systems with different singular geometries, for instance when the critical manifold features a fold \cite{krupa2001extending}. 

In this paper, we focus on the following extension of \eqref{2d}:
\begin{subequations}\eqlab{3d}
\begin{align}
    x' &= \varepsilon, \eqlab{3d-a}\\
    z_1' &=Z_1(x,z_1,z_2;\varepsilon), \eqlab{3d-b}\\
    z_2' &=Z_2(x,z_1,z_2;\varepsilon), \eqlab{3d-c}
\end{align}
\end{subequations}
where $Z_1$ and $Z_2$ are assumed to be sufficiently regular, that is, $C^\infty$-smooth in all of their arguments for simplicity, and where the critical manifold is given by $ \mc{C}_0 = \lb z_1=z_2=0\rb $, and is invariant for $ \varepsilon>0 $.
The linearisation of the fast subsystem \{\eqref{3d-b},\eqref{3d-c}\} along $\mc{C}_0$ is two-dimensional; we write $ A(x;\varepsilon) $ for its Jacobian matrix.
In the singular limit of $ \varepsilon=0 $, we denote the eigenvalues of $A(x;0)$ by $ \xi^{\pm}(x;0) $.
The critical manifold $\mc{C}_0$ is then normally hyperbolic where
\begin{align*}
    \Re\lb \xi^+(x;0), \xi^-(x;0)\rb \neq 0, 
\end{align*}
with $\Re\lb \cdot \rb$ denoting the real part of its argument. 

\begin{figure}[ht]\centering
	\begin{subfigure}[b]{0.5\textwidth}
	\centering
		\centering
		\begin{tikzpicture}
		\node at (0,0){
			\includegraphics[width=1\textwidth]{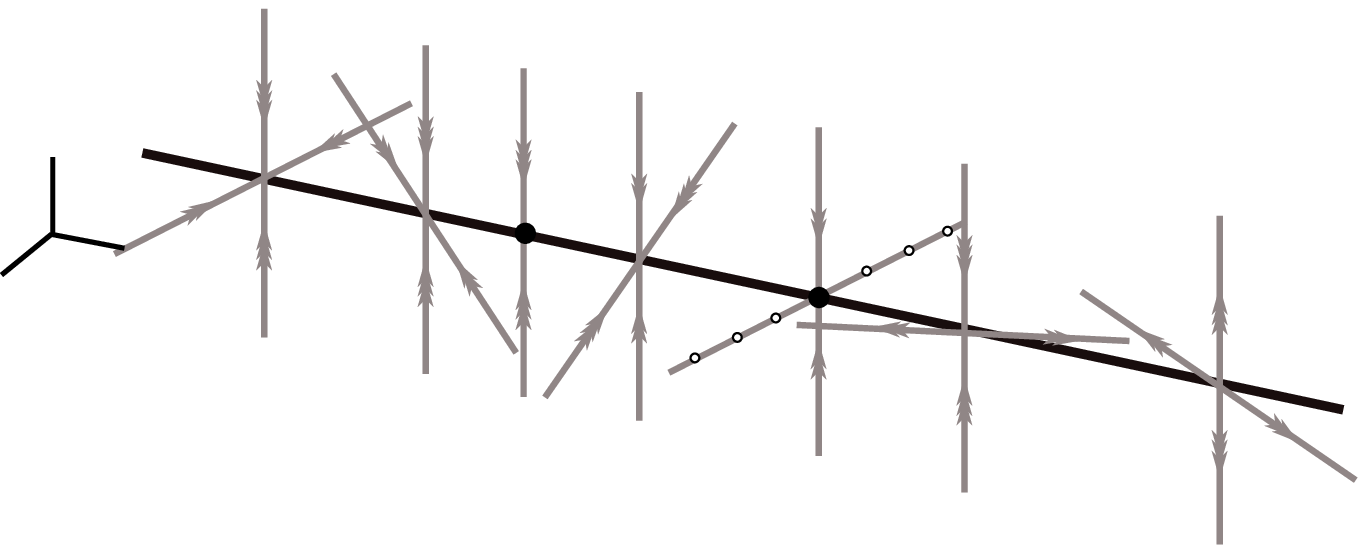}};
		\node at (-3.17,0.09) {$x$};
		\node at (-3.8,1) {$z_1$};
		\node at (-4.3,-0.13) {$z_2$};
		\node at (1,1.125) {$x=x^+$};
		\node at (-0.84,1.5) {$x=x^*$};
		\node at (-3.15,1) {$\mc{C}_0$};
		\end{tikzpicture}
		\caption{}
		\figlab{3dsf_a}
	\end{subfigure}
	~
	\begin{subfigure}[b]{0.45\textwidth}
	\centering
		\begin{tikzpicture}
		\node at (0,0){
			\includegraphics[scale = 0.6]{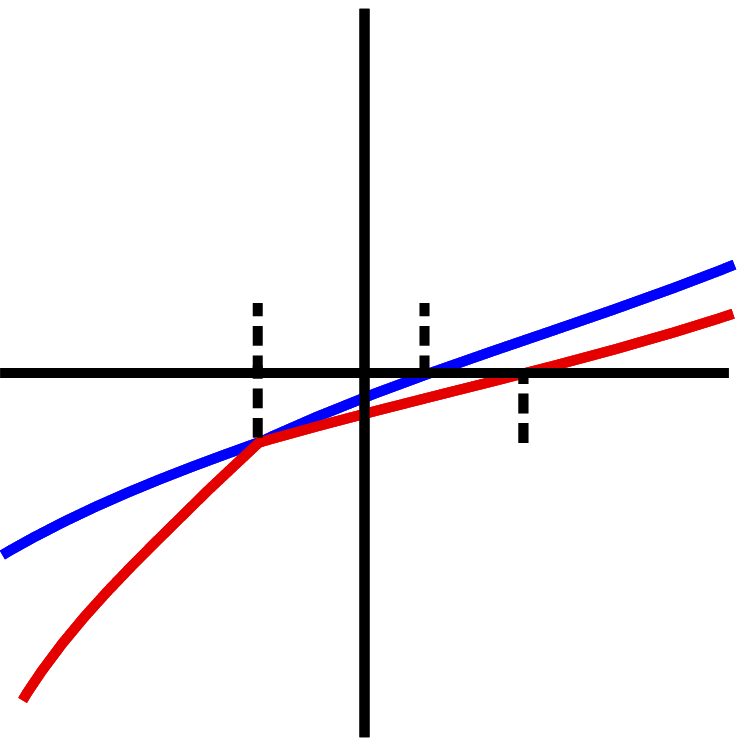}};
		\node at (2.4,0.0) {$x$};
		\node at (1.05,-0.6) {$x=x^-$};
		\node at (0.75,0.72) {$x=x^+$};
		\node at (-0.7,0.7) {$x=x^*$};
		\node at (-1.7,-0.33) {$\textcolor{blue}{\xi^+(x;0)}$};
		\node at (-1.,-1.8) {$\textcolor{red}{\xi^-(x;0)}$};
		\end{tikzpicture}
		\caption{}
		\figlab{3dsf_b}
	\end{subfigure}
	\caption{(a) Critical manifold $\mc{C}_0$ on which the fast subsystem has an improper node at some $x=x^*<x^+$, and a centre subspace at $x=0$. (b) Real eigenvalues $\xi^{\pm}(x;0)$ for the Jacobian matrix $A$ about $\mc{C}_0$: $\xi^+(x^*;0) = \xi^-(x^*;0)$, where the unique eigenvalue has geometric multiplicity $1$ at $x=x^*$. Moreover, $\xi^+(x;0)>0$ for $x>{x^+}$ and $\xi^-(x;0)>0$ for $x>{x^-}$; delayed loss of stability in this setting is studied in \secref{general}.}
	\figlab{3dsf}
\end{figure}

Specifically, we will be interested in the case where the eigenvalues $ \xi^\pm(x;0) $ are real\footnote{We refer the reader to \cite{baer1989slow,benoit2006dynamic,hayes2016geometric,neishtadt1987persistence,neishtadt1988persistence,su1993delayed} for results on the case where $\xi^\pm(x;0)$ are complex conjugates and \eqref{3d} passes through a Hopf bifurcation of the fast subsystem \eqref{3d-b}-\eqref{3d-c}.} and negative for $ x<x^* $, where $ x^*\in\mb{R} $, and where, moreover, $ \xi^-(x^*;0) = \xi^+(x^*;0) $, i.e., where the eigenvalues ``collide'' at $x=x^*$. Importantly, we assume that this collision occurs at a point where the accumulated contraction and expansion have not been balanced in either eigendirection individually, in the sense of equation~\eqref{wio-pl1}; these ideas will be made more precise in \secref{general} below. 
Moreover, after their ``collision''  at a point with $x = x^*$, at least one of the eigenvalues becomes positive as $ x $ increases. That is, for a given trajectory that enters a $ \delta $-neighbourhood of $ \mc{C}_0 $  with $ x<x^* $, contraction and expansion are accumulated as $ x $ increases until the trajectory exits the $ \delta $-neighbourhood when contraction and expansion are balanced, in analogy to \eqref{wio-pl1}. However, as at $x = x^*$ the unique eigenvalue has algebraic multiplicity $2$, and typically geometric multiplicity $1$, the two attracting eigendirections in the fast subsystem \eqref{3d-b}-\eqref{3d-c} are not linearly independent, with the corresponding point on $\mc{C}_0$ at $x = x^*$ an improper node for that subsystem. The resulting interaction between the subspaces of the linearisation about $\mc{C}_0$ with varying $ x $, which is illustrated in \figref{3dsf}, makes an extension of the known entry-exit relation in \eqref{wio-pl1} not straightforward. To the best of our knowledge, delayed loss of stability in this setting has not been studied before. 

We will, therefore, propose novel, extended formulae for the entry-exit relation for \eqref{3d}, in analogy to equation~\eqref{wio-pl1} for \eqref{2d}, in order to calculate the exit point in the above setting. To that end, we will first express the fast subsystem \eqref{3d-b}-\eqref{3d-c} in polar coordinates, exposing a two-dimensional fast-slow system of the form of
\begin{align*}
x' = \varepsilon, \quad
\theta' = \Phi(x,\theta,\varepsilon),
\end{align*}
in which $x$ is again the slow variable and the angular coordinate $\theta$ is the fast variable; crucially, equilibria of the $\theta$-equation correspond to angles of the eigendirections in the fast system \{\eqref{3d-b},\eqref{3d-c}\}. Then, depending on the properties of the critical manifold of this auxiliary system, i.e., on whether that manifold features a transcritical singularity \cite{krupa2001trans} or whether it contains a portion which is invariant for $\varepsilon>0$ \cite{schecter1985persistent}, we track the eigenspaces that the trajectories ``choose'' to follow for varying $x$, which allows us to construct the entry-exit relation for each of these cases. 

The paper is organised as follows. In \secref{general}, we formulate our main results for the general setting of equation~\eqref{3d} in the form of \thmref{thrm}. In \secref{sys}, we introduce a simple example, a system with one-way coupling, to demonstrate our methodology. Then, we modify that system to include an $\varepsilon$-dependence in the vector field, and we show that the conclusions reached differ from the $\varepsilon$-independent case, as predicted by \thmref{thrm}. {Finally, in the same section, we include an example with added nonlinearities in the vector field of our $\varepsilon$-dependent system, and we show that, in terms of delayed loss of stability, the behaviour of the latter is similar to that of our linear example.} We conclude the paper in \secref{conc}. 




\section{Extended entry-exit formula}\seclab{general}

In this section, we derive our main result, \thmref{thrm}; to that end, we first formulate a number of underlying assumptions. Crucially, we transform equation~\eqref{3d} to cylindrical coordinates, which will allow us to describe naturally the dynamics near $x=x^\ast$.

\subsection{Main assumptions}
We consider systems of the form
\begin{subequations}\eqlab{gen}
	\begin{align}
	{x}'&=\varepsilon,\eqlab{gen-a}\\
	{z}_1'&=Z_1(x,z_1,z_2,\varepsilon),\eqlab{gen-b}\\
	{z}_2'&=Z_2(x,z_1,z_2,\varepsilon),\eqlab{gen-c}
	\end{align}
\end{subequations}
where the functions $Z_1$ and $Z_2$ are $C^\infty$-smooth in all of their arguments and the corresponding critical manifold is now given by
\begin{align*}
    \mc{C}_0= \lb (x,z_1,z_2)\in \mb{R}^3~\lvert~z_1=0=z_2\rb.
\end{align*}
The linearisation of \eqref{gen} about $\mc{C}_0$ reads
\begin{subequations}\eqlab{genmat}
	\begin{align}
	\begin{matrix}
	{x}
	\end{matrix}'&=\varepsilon,\\
	\begin{pmatrix}
	{z}_1 \\
	{z}_2
	\end{pmatrix}'
	&=
	A(x;\varepsilon)
	\begin{pmatrix}
	z_1 \\ z_2
	\end{pmatrix}, \eqlab{zsub}
	\end{align} 
\end{subequations}
where 
\begin{align}
A(x;\varepsilon) = \begin{pmatrix}
f_1(x;\varepsilon) & f_2(x;\varepsilon)\\
g_1(x;\varepsilon) & g_2(x;\varepsilon)
\end{pmatrix} \eqlab{amat}
\end{align}
with
\begin{gather*}
    f_1(x;\varepsilon) = \dfrac{\partial Z_1}{\partial{z_1}}(x,0,0;\varepsilon), \quad  f_2(x;\varepsilon) = \dfrac{\partial Z_1}{\partial{z_2}}(x,0,0;\varepsilon),
    \\
    g_1(x;\varepsilon) = \dfrac{\partial Z_1}{\partial{z_2}}(x,0,0;\varepsilon), \quad\text{and}\quad  g_2(x;\varepsilon) = \dfrac{\partial Z_2}{\partial{z_2}}(x,0,0;\varepsilon).
\end{gather*}
We denote the eigenvalues of the matrix $ A(x; \varepsilon) $ by
\begin{align}
\xi^{\pm}(x;\varepsilon) =\frac{1}{2}\lp  \tn{tr}{A}(x,\varepsilon)\pm\sqrt{\lp \tn{tr}{A}(x,\varepsilon)\rp^2-4\lp \tn{det}{A}(x,\varepsilon)\rp}\rp; \eqlab{eigs}
\end{align}
alternatively, we may denote them by
\begin{align}
\mu_1(x; \varepsilon) : = \begin{cases}
\xi^+(x; \varepsilon) \quad \tn{if } x<x^*  \\
\xi^-(x; \varepsilon) \quad \tn{if } x>x^* 
\end{cases} 
\quad\text{and}\quad
\mu_2(x; \varepsilon) : = \begin{cases}
\xi^-(x; \varepsilon) \quad \tn{if } x<x^*  \\
\xi^+(x; \varepsilon) \quad \tn{if } x>x^* 
\end{cases};
\eqlab{mus}
\end{align}
see \figref{gen-eigs} for an illustration. We note that the representation in \eqref{eigs} is potentially only $C^0$-smooth, i.e., continuous, at $x=x^*$; however, it has the advantage of $\xi^+$ always being the ``dominant" eigenvalue.


\begin{figure}[H]\centering

    \begin{subfigure}[b]{0.3\textwidth}
	\centering
			\begin{tikzpicture}
	\node at (0,0){
		\includegraphics[scale = 0.6]{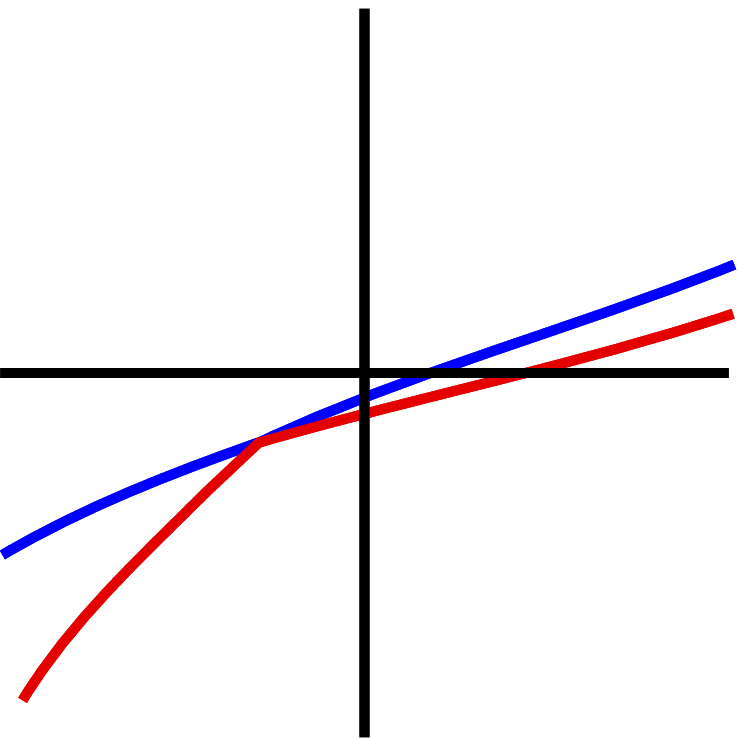}};
	
	\node at (2.4,0.0) {$x$};
	
	\node at (1.7,1)
	{\small \textcolor{blue}{$\xi^+(x;0)$}};
	\node at (-1.15,-1.75) {\small \textcolor{red}{$\xi^-(x;0)$}};

	\node at (-0.8,0.8) {$\mu_1(x;0)$};
	\draw[-Stealth] (-1.2,0.5) -- (-1.2,-0.5);
	\draw[-Stealth] (-0.15,0.8) -- (2.15,0.38);
	
	\node at (1.5,-1.5) {$\mu_2(x;0)$};
	\draw[-Stealth] (1.3,-1.2) -- (1.3,0.35);
	\draw[-Stealth] (0.8,-1.5) -- (-1.3,-1.1);
	\end{tikzpicture}
		\caption{}
	\end{subfigure}
	\begin{subfigure}[b]{0.3\textwidth}
	\centering
			\begin{tikzpicture}
	\node at (0,0){
		\includegraphics[scale = 0.6]{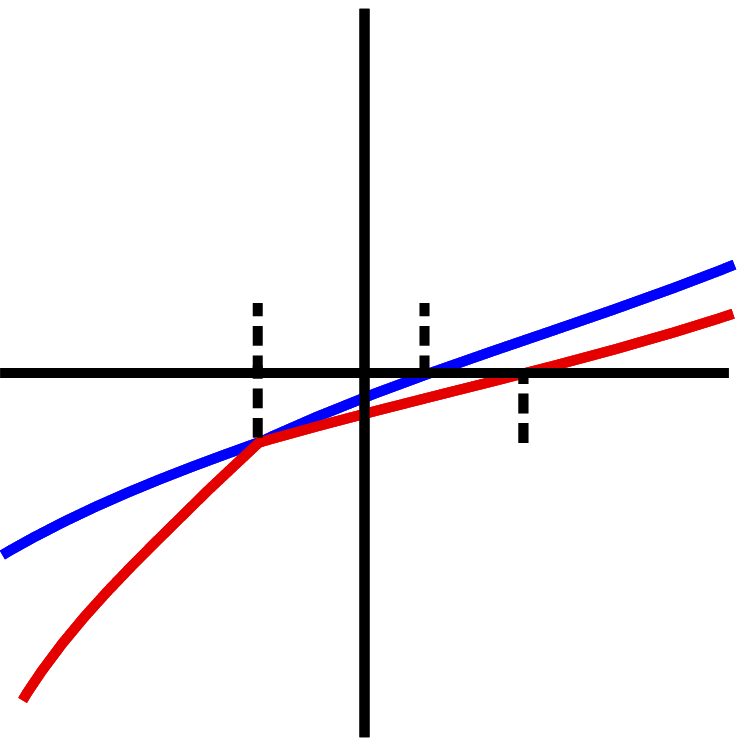}};
	
	\node at (2.4,0.0) {$x$};
	\node at (1.05,-0.6) {$x=x_-$};
	\node at (0.75,0.6) {$x=x_+$};
	\node at (-0.7,0.7) {$x=x^*$};
	
	\end{tikzpicture}
		\caption{}
	\end{subfigure}
	\begin{subfigure}[b]{0.3\textwidth}
		\centering
			\begin{tikzpicture}
	\node at (0,0){
		\includegraphics[scale = 0.6]{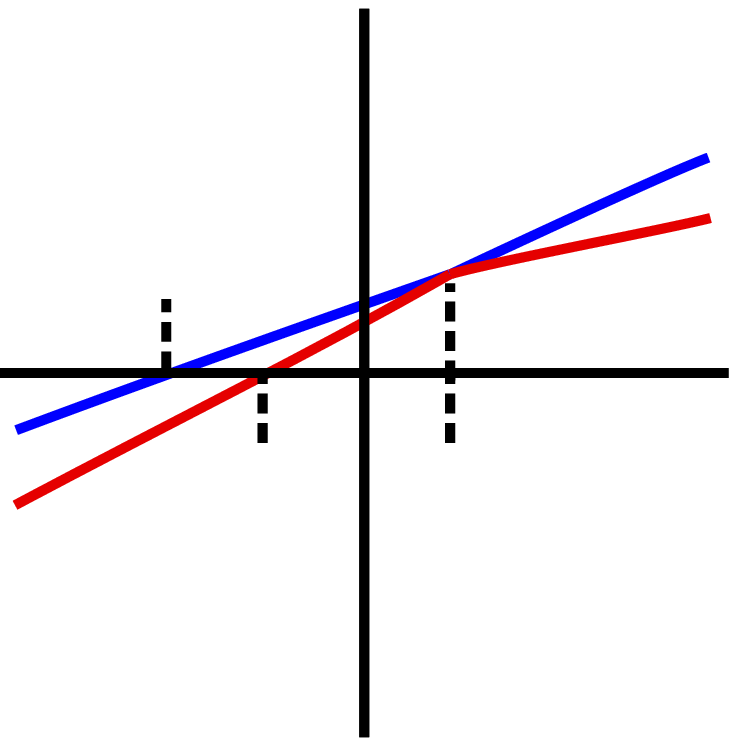}};
	\node at (2.4,0.0) {$x$};
	\node at (-0.7,-0.65) {$x=x_-$};
	\node at (0.69,-0.55) {$x=x^*$};
	\node at (-1.1,0.62) {$x=x_+$};
	\end{tikzpicture}
		\caption{}
	\end{subfigure}
	\caption{(a) The eigenvalues of the matrix $A(x; 0)$ in \eqref{amat} can be expressed either via the functions $\xi^{\pm}$ as in equation~\eqref{eigs}, where $\xi^{+}$ is always above $\xi^{-}$, or in the form of $\mu_{i}$ ($i=1,2$) as in equation~\eqref{mus}; (b) $x^*<0$ and $x_\pm>0$; (c) $x^*>0$ and $x_\pm<0$.}
	\figlab{gen-eigs}
\end{figure}
In this paper, we are concerned with the scenario where the following set of assumptions is satisfied:
\begin{asu}\asulab{mus}
We consider an interval $ I\subset \mb{R} $, and we assume the following.
\begin{enumerate}
    \item The critical manifold $\mc{C}_0$ is invariant for all $\varepsilon>0$. 
	\item The eigenvalues $\xi^{\pm}(x;0)$ in \eqref{eigs} are real and non-decreasing for $x\in I$.
	\item There exists $ x_+ \in I$ such that $\xi^+(x_+;0)=0$ and/or $ x_- \in I$ such that $\xi^-(x_-;0)=0$; if $x_+$ and/or $x_-$ exist, they are unique.
	\item There exists a unique $ x^*\in I$ such that $\xi^+( x^*;0)=\xi^-( x^*;0)$ holds.
	\item It holds that $ x^*<\min\lb x_1^{(1)}, x_1^{(2)}\rb $, where $ x_1^{(1)}$ and $ x_1^{(2)} $ are such that
	\begin{align*}
	\int_{x_0}^{x_1^{(1)}} \mu_1(x;0)\tn{d}x =0 \quad\text{and}\quad \int_{x_0}^{x_1^{(2)}} \mu_2(x;0)\tn{d}x =0.
	\end{align*}
	If one of the above two integral equations has no solution, we take $x_i^{(1)}=+\infty$ ($i=1,2$).
	\item The eigenvalue $\xi^* := \xi^+(x^*;0)=\xi^-(x^*;0)$ of the matrix $A(x^*;0)$ has geometric multiplicity $1$.
\end{enumerate}
\end{asu}

Note that the first item in \asuref{mus} above implies that the critical manifold $\mc{C}_0$ has no folds in the sense of \cite{krupa2001extending}. From the second and third items, it follows that $\mc{C}_0$ can be decomposed into attracting and repelling branches, as follows:
\begin{gather*}
    \mc{C}_0^a= \lb (x,z_1,z_2)\in \mc{C}_0~\lvert~\xi^\pm(x; \varepsilon)<0\rb\quad\text{and}
    \\
    \mc{C}_0^r= \lb (x,z_1,z_2)\in \mc{C}_0~\lvert~\xi^+(x; \varepsilon)>0\ \tn{or}\ \xi^-(x; \varepsilon)>0\rb.
\end{gather*}
It is important for our analysis that the fast subsystem in \eqref{zsub} undergoes no Hopf bifurcations along $\mc{C}_0$; from the analysis in \cite{kaklamanos2022bifurcations,letson2017analysis}, it follows that if either one of the manifolds given by $Z_{1}(x,z_1,z_2;0)=0$ or $Z_{2}(x,z_1,z_2;0)=0$ has a fold line, then $\mc{C}_0$ has a fold point, with the fast subsystem \eqref{zsub} undergoing a Hopf bifurcation close to that point. Moreover, from the third point in \asuref{mus}, each eigenvalue can become zero at at most one point which excludes, for instance, the case where one of the eigenvalues is constant and zero.

By the last three items in \asuref{mus}, at the point $x=x^*$ where the eigenvalues $\xi^\pm$ intersect, the two corresponding eigenspaces ``collide'' into one, at a point where contraction and expansion have not been balanced along each eigendirection individually: in particular, these eigenvalues can attain either negative or positive values at their intersection; recall panels (a) and (b) of \figref{gen-eigs}, respectively. We are therefore interested in how this interaction between the eigenspaces affects the overall dynamics of the system, in terms of its implications about the accumulated contraction and expansion and an entry-exit relation analogous to the one in \eqref{wio-pl1}. If at $x=x^*$ the unique eigenvalue $\xi^*$ has geometric multiplicity 2, then the system is globally diagonalisable, and delayed loss of stability can be studied along each eigenspace separately.

Finally, our analysis here is local and focused on the phenomenon of delayed loss of stability along $\mc{C}_0$; higher-order nonlinearities do not contribute, but can potentially play the role of a return mechanism that re-injects trajectories onto the attracting portion $\mc{C}^a_0$ of $\mc{C}_0$, forming closed trajectories that contain plateau segments; see \cite{de2016entry,desroches2012mixed,kaklamanos2022bifurcations,kuehn2011decomposing,sadhu2019complex} for examples of return mechanisms in three-dimensional fast-slow systems.

\subsection{Polar coordinates and ``hidden'' dynamics}
The interaction and collision of eigendirections described above can be more easily studied by transforming the fast subsystem in \eqref{zsub} into polar coordinates, which corresponds to the full system in \eqref{genmat} being written in cylindrical coordinates:
\begin{lemma} \lemmalab{polars}
	In cylindrical coordinates $ (x, r, \theta) $, with $x=x$, $z_1=r \cos\theta$, and $z_2=r \sin\theta$, equation~\eqref{genmat} reads 
	\begin{subequations}\eqlab{3pc}
		\begin{align}
		x' &= \varepsilon, \eqlab{3pc-a}\\
		r' &= \lp f_1(x;\varepsilon)\cos\theta+f_2(x;\varepsilon)\sin\theta\rp r\cos\theta+\lp g_1(x;\varepsilon)\cos\theta+g_2(x;\varepsilon)\sin\theta\rp r \sin\theta, \eqlab{3pc-b}\\
		\theta' &= f_2(x;\varepsilon)\sin^2\theta+g_1(x;\varepsilon)\cos^2\theta+\lp g_2(x;\varepsilon)-f_1(x;\varepsilon)\rp\cos\theta\sin\theta. \eqlab{3pc-c}
		\end{align}
	\end{subequations}
\end{lemma}
\begin{proof}
	Direct calculations.
\end{proof}
Note that the vector field in \eqref{3pc} is periodic in $\theta$, with period $\pi$. Hence, the $\theta$-variable therein can be restricted to the interval $\bigg[-\dfrac{\pi}{2},\dfrac{\pi}{2}\bigg)$, with values outside that interval taken modulo $\pi$. In the following, we will denote the right-hand side in \eqref{3pc-c} by
\begin{align*}
\Phi(x,\theta,\varepsilon) := f_2(x;\varepsilon)\sin^2\theta+g_1(x;\varepsilon)\cos^2\theta+\lp g_2(x;\varepsilon)-f_1(x;\varepsilon)\rp\cos\theta\sin\theta, 
\end{align*}
for which, for future reference, we have
\begin{align}
\begin{aligned}
\frac{\partial \Phi}{\partial \theta}(x,\theta,\varepsilon) =& \lp g_2(x;\varepsilon)-f_1(x;\varepsilon)\rp\cos^2\theta+2\lp f_2 (x;\varepsilon)-g_1 (x;\varepsilon)\rp\sin\theta\cos\theta\\
&+\lp f_1(x;\varepsilon)-g_2(x;\varepsilon)\rp\sin^2\theta.
\end{aligned}\eqlab{phix}
\end{align}
We observe that \eqref{3pc-a} and \eqref{3pc-c} are decoupled from \eqref{3pc-b}, and that the set $\{r=0\}$, corresponding to the critical manifold $\mc{C}_0$, is invariant. Since we are interested in delayed loss of stability along $\mc{C}_0$, we will hence restrict our analysis to $r=0$, and we will focus on the system
\begin{subequations}\eqlab{pc0}
	\begin{align}
	x' &= \varepsilon, \eqlab{pc-a}\\
	\theta' &= \Phi(x,\theta;\varepsilon), \eqlab{pc-b}
	\end{align}
\end{subequations}
{which} is a two-dimensional fast-slow system in the standard form of GSPT. In terms of the variables $(x,\theta)$, the critical manifold for \eqref{pc0} reads
\begin{align}
\mc{M}_0 = \lb (x,\theta)\in \mathbb{R}\times \lp\mb{R}\ \tn{mod}\ \pi\rp\ ~\lvert~ \Phi(x,\theta,0)=0\rb. \eqlab{crima}
\end{align}
\begin{lemma}
    The scalar problem \eqref{pc-b}$_{\varepsilon=0}$, given by $\theta' = \Phi(x,\theta;0)$, undergoes a transcritical bifurcation at $(x,\theta) = (x^*,\theta^*)$. 
    \lemmalab{tc}
\end{lemma}
\begin{proof}
Expanding the function $\Phi(x,\theta;0)$ about $\theta = \theta^*$ gives
\begin{align}
\Phi(x,\theta;0) = T_1(x)(\theta-\theta^*)+T_2(x)(\theta-\theta^*)^2+\dots, \eqlab{tcnf}
\end{align}
where the dots denote higher-order terms in $\theta$ and, moreover,
\begin{align*}
    T_1(x) :&= (g_2(x,0)-f_1(x,0))\left(\cos^2\theta^*-\sin ^2\theta^*\right) +2 \lp f_2(x,0)-g_1(x,0)\rp\sin \theta^* \cos \theta^*\quad\text{and} \\
    T_2(x) :&= \lp f_2(x,0) -g_1(x,0)\rp\left(\cos ^2\theta^*-\sin ^2\theta^*\right)-2 (g_2(x,0)-f_1(x,0))\sin \theta^* \cos \theta^* 
\end{align*}
The expression in \eqref{tcnf} is the normal form of a transcritical bifurcation.
\end{proof}

\lemmaref{tc} implies that the critical manifold $\mc{M}_0$ consists of two branches, $\mc{S}_0$ and $\mc{Z}_0$,
which intersect and exchange stability at $(x,\theta)=(x^*,\theta^*)$; cf. \figref{stabex}. Indeed,
for any fixed $ x $, the $ \theta$-roots of $ \Phi(x,\theta,0)=0 $ correspond to the angles of the eigenvectors of $ A(x;0) $ with the positive $x$-axis. Moreover, for $x\neq x^*$, the matrix $ A(x;0) $ has two distinct eigenvalues, and hence two distinct eigendirections. In terms of their angles $\theta$, the latter can be represented as graphs over $x$ in the $(x,\theta)$-plan, as a consequence of the implicit function theorem; by assumption, the two eigenvalues exist for every $x\in \mathbb{R}$, and so do their directions, represented by the angular coordinate $\theta$. We therefore denote by $\theta_{i}(x)$ the angle of the eigendirection associated with $\mu_{i}(x;0)$, where $i=1,2$; recall \eqref{mus}. Correspondingly, we define the branches of $\mc{M}_0 = \mc{S}_0\cup \mc{Z}_0$ by
\begin{align*}
    \mc{S}_0 : \begin{cases}
     \theta = \theta_{1}(x) &\tn{if} \quad x\neq x^* \\
     \theta =\theta^* &\tn{if} \quad x= x^*
    \end{cases}
    \quad\text{and}\quad
    \mc{Z}_0 : \begin{cases}
     \theta = \theta_{2}(x) &\tn{if} \quad x\neq x^* \\
     \theta =\theta^* &\tn{if} \quad x= x^*
    \end{cases}
    ,
\end{align*}
respectively.

\begin{figure}[H]\centering
	\begin{subfigure}[b]{0.45\textwidth}
	\centering
			\begin{tikzpicture}
	\node at (0,0){
		\includegraphics[width=0.9\textwidth]{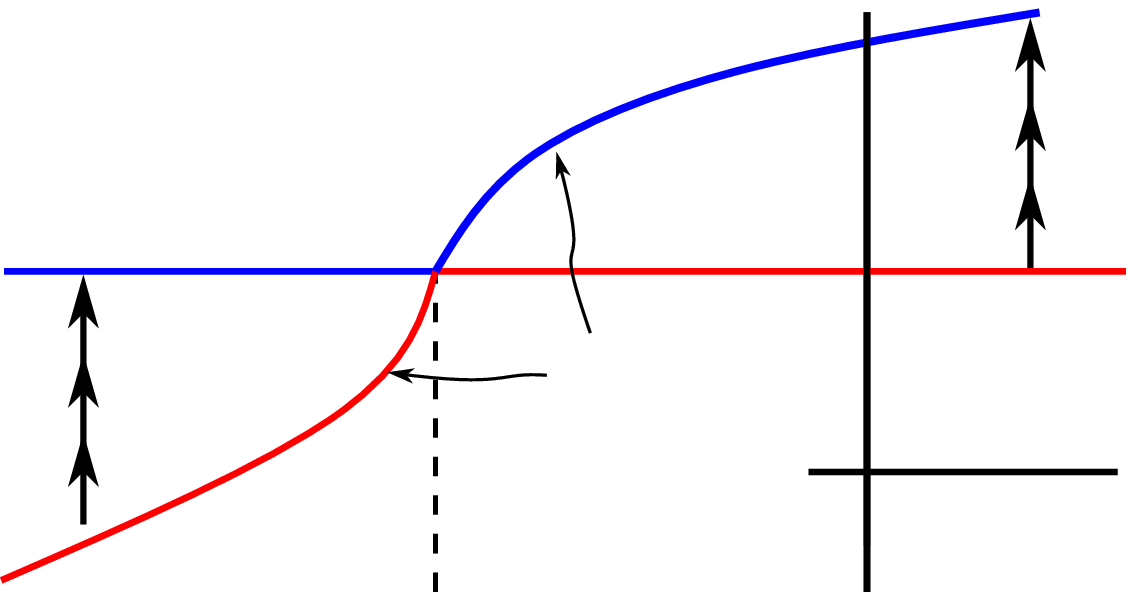}};
	\node at (3.5,-1) {$x$};
	\node at (1.8,2) {$\theta$};
	\node at (-0.75,-2) {$x=x^*$};
	\node at (-3,0.5) {$\mc{S}_0$};
	\node at (0.2,-0.5) {$\mc{Z}_0$};
	\end{tikzpicture}
		\caption{	\figlab{stabex_a}}
	\end{subfigure}
~
	\begin{subfigure}[b]{0.45\textwidth}
		\centering
			\begin{tikzpicture}
	\node at (0,0){
		\includegraphics[width=0.9\textwidth]{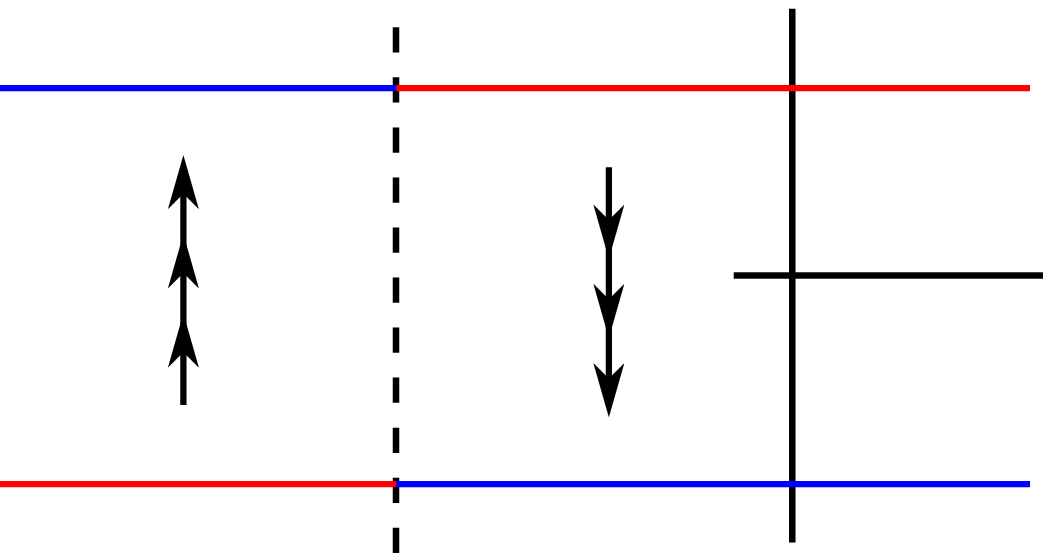}};
	\node at (-3,1.5) {$\mc{S}_0$};
	\node at (3.5,0) {$x$};
	\node at (1.8,2) {$\theta$};
	\node at (-0.75,-2) {$x=x^*$};
	\node at (-3,-1) {$\mc{Z}_0$};
	\end{tikzpicture}
		\caption{\figlab{stabex_b}}
	\end{subfigure}
	\caption{At $x=x^*$, the critical manifold $\mc{M}_0$ loses normal hyperbolicity. If the eigenvalue of the matrix $A(x^*;0)$ has geometric multiplicity $1$, then $\mc{M}_0$ has a self-intersection, as in panel (a). If the eigenvalue of the matrix $A(x^*;0)$ has geometric multiplicity $2$, then $\mc{M}_0$ consists of two disjoint branches, as in panel (b).}
	\figlab{stabex}
\end{figure}

Regarding the stability properties of $\mc{S}_0 $ and $\mc{Z}_0 $, we have the following result:
\begin{lemma}
	Fix $ x\neq x^* $. Then, for the scalar problem \eqref{pc-b}$_{\varepsilon=0} $, the branch $\mc{S}_0$ of the critical manifold $\mc{M}_0$ is attracting (repelling) if $x<x^*$ ($x>x^*$), whereas the branch $\mc{Z}_0$ is repelling (attracting) if $x<x^*$ ($x>x^*$).
	\lemmalab{stabeig}
\end{lemma}
\begin{proof}
	For any $ x\neq x^* $, equation~\eqref{amat} is diagonalisable, and there exists a change of coordinates such that \eqref{zsub} can be written as
	\begin{align}
	\begin{pmatrix}
	\zeta_1 \\ \zeta_2 
	\end{pmatrix}'
	=\begin{pmatrix}
	\mu_1(x;0) & 0 \\
	0 &	\mu_2(x;0) 
	\end{pmatrix}\begin{pmatrix}
	\zeta_1 \\ \zeta_2 
	\end{pmatrix}. \eqlab{diaged}
	\end{align}
In these coordinates, the eigenvector associated with $ \mu_1 $ is $ (1,0) $, with angle $ \omega_1 = 0 $, corresponding to $ \theta_1 $ in the original coordinates; similarly, the eigenvector associated with $ \mu_2 $ is $ (0,1) $, with angle $ \omega_2 = \frac{\pi}{2} $, corresponding to $ \theta_2 $.

In the notation of \eqref{zsub} and \eqref{amat}, for \eqref{diaged} we have
\begin{align*}
f_1(x;0) = \mu_1(x;0), \quad f_2(x;0) =0, \quad g_1(x;0) = 0, \quad\text{and}\quad g_2(x;0) = \mu_2(x;0). 
\end{align*}
In polar coordinates, it follows that
\begin{subequations}
\begin{align*}
\Phi(x,\omega,0) &= \lp \mu_2(x;0)-\mu_1(x;0)\rp\cos\omega\sin\omega, \\
\frac{\partial \Phi}{\partial \omega}(x,\omega,0) &= \lp \mu_2(x;0)-\mu_1(x;0)\rp\cos^2\omega+\lp \mu_1(x;0)-\mu_2(x;0)\rp\sin^2\omega.
\end{align*}
\end{subequations}
It is therefore apparent that $ \frac{\partial \Phi}{\partial \omega}(x,\omega_1,0)<0 $ and $\frac{\partial \Phi}{\partial \omega}(x,\omega_2,0)>0 $ when $x<x^*$, whereas $ \frac{\partial \Phi}{\partial \omega}(x,\omega_1,0)>0 $ and $\frac{\partial \Phi}{\partial \omega}(x,\omega_2,0)<0 $ for $x>x^*$, which gives the result.
\end{proof}
In summary, we may therefore write $\mc{S}_0 = \mc{S}_0^a\cup\lb (x^*,\theta^*)\rb\cup\mc{S}_0^r$ and $\mc{Z}_0 = \mc{Z}_0^r\cup\lb (x^*,\theta^*)\rb\cup\mc{Z}_0^a$, where
\begin{subequations}\eqlab{branchez}
\begin{gather}
    \mc{S}_0^a = \lb (x,\theta)\in \mc{S}_0 ~\lvert~ x<x^* \rb, \quad
    \mc{S}_0^r = \lb (x,\theta)\in \mc{S}_0 ~\lvert~ x>x^* \rb, \\
    \mc{Z}_0^r = \lb (x,\theta)\in \mc{Z}_0 ~\lvert~ x<x^* \rb, \quad\text{and}\quad
    \mc{Z}_0^a = \lb (x,\theta)\in \mc{Z}_0 ~\lvert~ x>x^* \rb.
\end{gather}
\end{subequations}

Finally, for future reference, we highlight some of the properties of the function $\Phi(x,\theta,\varepsilon)$ which follow from the transcritical bifurcation at $(x,\theta,\varepsilon)=(x^*,\theta^*,0)$ in \eqref{pc0}:
\begin{cor} Recall the definition of \eqref{pc0} and \lemmaref{tc}. Then, given \asuref{mus}, the following relations hold: 
\begin{subequations}
\begin{align}
\frac{\partial \Phi}{\partial \theta} (x^*, \theta^*;0) &= 0,\\
\frac{\partial \Phi}{\partial x} (x^*, \theta^*;0) &= 0,\\
\frac{\partial^2 \Phi}{\partial \theta^2} (x^*, \theta^*;0) &\neq 0,\\
\dfrac{\partial^2 \Phi}{\partial x \partial \theta}\lp x^*, \theta^*, 0\rp &\neq0,\quad\text{and} \\
\left|\begin{matrix}
\dfrac{\partial^2 \Phi}{\partial \theta^2}\lp x^*, \theta^*, 0\rp & \dfrac{\partial^2 \Phi}{\partial x \partial \theta}\lp x^*, \theta^*, 0\rp\\
\dfrac{\partial^2 \Phi}{\partial x \partial \theta}\lp x^*, \theta^*, 0\rp & \dfrac{\partial^2 \Phi}{\partial x^2}\lp x^*, \theta^*, 0\rp
\end{matrix}\right| &<0.
\eqlab{phipart}
\end{align} 
\end{subequations}
\end{cor}
\begin{proof}
	The statements follow from the fact that \eqref{pc-b}$_\varepsilon=0$ undergoes a transcritical bifurcation at $(x,\theta) = (x^*,\theta^*)$; cf. \lemmaref{tc}.
\end{proof}


\subsection{Main result}
We now present our main result. Throughout, we assume that a trajectory of the original system, equation~\eqref{gen}, enters a $\delta$-neigbourhood of $\mc{C}_0$, say a cylinder $B_\delta$ of radius $\delta$ around $\mc{C}_0$, with $\delta>0$ small, at a point with $x=x_0$. Given that the eigenvalues of the linearisation about $\mc{C}_0$ behave in accordance with \asuref{mus}, our aim is to find formulae that indicate how the accumulated contraction to, and expansion from, $\mc{C}_0$ can be balanced in order to calculate the exit coordinate $x_1$ at which the aforementioned trajectory exits the $\delta$-cylinder $B_\delta$ about $\mc{C}_0$. 
\begin{theorem}\thmlab{thrm}
Assume that \asuref{mus} holds for equation~\eqref{genmat}, and that a trajectory of \eqref{gen} enters a $\delta$-cylinder $B_\delta$ about $\mc{C}_0$, for $\delta>0$ sufficiently small, at some point with $x = x_0<x^*$. Denote
\begin{gather}
\begin{gathered}
\alpha := \frac{1}{2}\dfrac{\partial^2 \Phi}{\partial \theta^2}\lp x^*, \theta^*, 0\rp,\quad
\beta := \frac{1}{2}\dfrac{\partial^2 \Phi}{\partial x \partial \theta}\lp x^*, \theta^*, 0\rp,\\
\gamma := \frac{1}{2}\dfrac{\partial^2 \Phi}{\partial x^2}\lp x^*, \theta^*, 0\rp,\quad\text{and}\quad
\delta := \frac{1}{2}\dfrac{\partial \Phi}{\partial \varepsilon}\lp x^*, \theta^*, 0\rp, 
\end{gathered}\eqlab{kscoef}
\end{gather}
as well as
\begin{align}
\lambda := \dfrac{\delta\alpha +\beta}{\sqrt{\beta^2-\gamma\alpha}}. \eqlab{lam}
\end{align}
\begin{enumerate}
	\item If $ \lambda \neq 1 $, or if $ \lambda =1 $ and $\mc{Z}_0$ is invariant for \eqref{pc0} with $\varepsilon>0$, then the given trajectory exits $B_\delta$ at a point with $ x=x_1 +o(1)$, where $x_1$ is obtained by solving
	\begin{align}
	\int_{x_{0}}^{x^{*}} \mu_1(x) \tn{d}x +\int_{x^{*}}^{x_{1}} \mu_2(x) \tn{d}x =0. \eqlab{trans}
	\end{align}
	
	\item If $ \lambda = 1 $ and $\mc{S}_0$ is invariant for \eqref{pc0} with $\varepsilon>0$, then the trajectory exits $B_\delta$ at a point with $ x=x_1 + o(1)$, where $x_1$ is obtained by solving
	\begin{align}
	\int_{x_{0}}^{\tilde{x}} \mu_1(x) \tn{d}x  + \int_{\tilde{x}}^{x_1} \mu_2(x) \tn{d}x =0, \eqlab{invar}
	\end{align}
	and where, moreover, $ \tilde{x} $ is found by solving
	\begin{align}
		\int_{x_{0}}^{\tilde{x}}  \dfrac{\partial \Phi}{\partial \theta}\lp x, \theta_1(x), 0\rp \tn{d}x=0. \eqlab{xtil}
	\end{align}
\end{enumerate}
\end{theorem}
\begin{proof}
By \lemmaref{polars}, equation~\eqref{gen} gives a fast-slow system in polar coordinates of the form in \eqref{pc0}. The latter has a critical manifold $ \mc{M}_0 = \mc{S}_0\cup \mc{Z}_0 $, given by \eqref{crima}, where  $\mc{S}_0 = \mc{S}_0^a\cup\lb (x^*,\theta^*)\rb\cup\mc{S}_0^r$ and $\mc{Z}_0 = \mc{Z}_0^r\cup\lb (x^*,\theta^*)\rb\cup\mc{Z}_0^a$; recall \eqref{branchez}. The branches $\mc{S}_0$ and $\mc{Z}_0$ intersect transversely and exchange stability at $ (x^*,\theta^*) $; recall \lemmaref{stabeig} and see \figref{selfint} for an illustration.
\begin{enumerate}
	\item If $ \lambda\neq 1 $, then by \cite{krupa2001extending}, \eqref{pc0} can be locally written in the form of the transcritical singularity studied therein. Hence, for $ \varepsilon>0 $ sufficiently small, a trajectory of \eqref{pc0} that follows the attracting slow branch $ \mc{S}^a_\varepsilon $ will follow the attracting slow branch $ \mc{Z}^a_\varepsilon $ after crossing $ x=x^* $ regardless of whether $\lambda<1$ or $\lambda>1$, due to the equivalence relation implied by \eqref{crima}. If, on the other hand, $ \lambda =1 $ and $ \mc{Z}_0 $ is invariant for \eqref{pc0} with $\varepsilon>0$ small, then trajectories will again follow the attracting branch $ \mc{S}^a_\varepsilon $ for $x<x^*$ and the invariant attracting branch $ \mc{Z}^a_0 $ for $x>x^*$.
	
	The above implies that, in the original system in \eqref{gen} with $ \varepsilon>0 $, contraction is accumulated in the eigendirection of $ \theta_1(x) $ for $ x<x^* $, while contraction and expansion are accumulated in the eigendirection of $ \theta_2(x) $ for $ x>x^* $. The total contraction and expansion are therefore balanced in accordance with the entry-exit formula in \eqref{trans}.
	
	\item If $ \lambda =1 $ and $ \mc{S}_0 $ is invariant for \eqref{pc0} with $\varepsilon>0$ small, then for $ \varepsilon>0 $ sufficiently small, a trajectory of \eqref{pc0} that follows the attracting branch $ \mc{S}^a_0 $ will experience bifurcation delay after crossing $ x=x^* $ and before ``jumping'' to follow an attracting slow branch $ \mc{Z}^a_\varepsilon $. The exit point $ \tilde{x} $ is calculated via \eqref{xtil} \cite{schecter1985persistent}.
	
	The above implies that, in \eqref{gen} with $ \varepsilon>0 $, contraction is accumulated in the eigendirection of $ \theta_1(x) $ for $ x<\tilde{x} $, while contraction and expansion are accumulated in the eigendirection of $ \theta_2(x) $ for $ x>\tilde{x} $. The total contraction and expansion are therefore balanced in accordance with the entry-exit formula in \eqref{invar}.
\end{enumerate}
\end{proof}

\begin{figure}[H]\centering
	\begin{subfigure}[b]{0.45\textwidth}
	\centering
			\begin{tikzpicture}
	\node at (0,0){
		\includegraphics[width=0.9\textwidth]{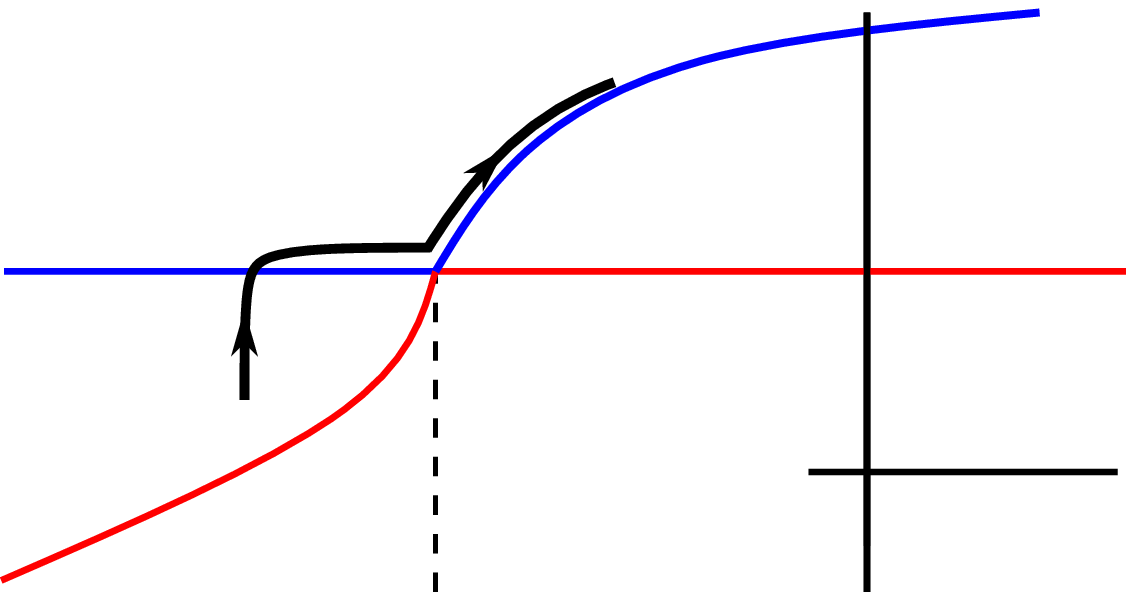}};
	
	\node at (3.5,-1) {$x$};
	\node at (-1.9,-0.75) {$x_0$};
	\node at (1.8,2) {$\theta$};
	\node at (-0.75,-2) {$x=x^*$};
	
	
	\node at (-3,0.5) {$\textcolor{blue}{\mc{S}_0^a}$};
	\node at (1.3,0.5) {$\textcolor{red}{\mc{S}_0^r}$};
	\node at (-2.5,-1.7) {$\textcolor{red}{\mc{Z}_0^r}$};
	\node at (1.3,1.8) {$\textcolor{blue}{\mc{Z}_0^a}$};
	\end{tikzpicture}
		\caption{$\lambda \neq 1$\figlab{selfint_a}}
	\end{subfigure}
	~
		\begin{subfigure}[b]{0.45\textwidth}	\centering
			\begin{tikzpicture}
	\node at (0,0){
		\includegraphics[width=0.9\textwidth]{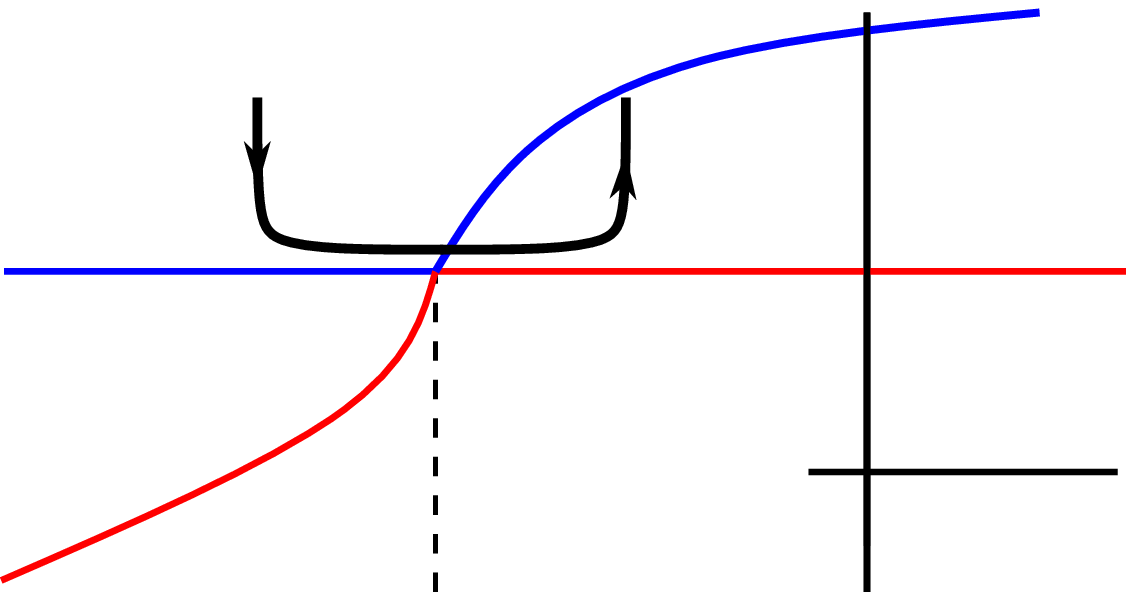}};
	
	\node at (3.5,-1) {$x$};
	\node at (-1.8,1.4) {$x_0$};
	\node at (0.3,1.5) {$\tilde{x}$};
	\node at (1.8,2) {$\theta$};
	\node at (-0.75,-2) {$x=x^*$};
	
    
    \node at (-3,0.5) {$\textcolor{blue}{\mc{S}_0^a}$};
	\node at (1.3,0.5) {$\textcolor{red}{\mc{S}_0^r}$};
	\node at (-2.5,-1.7) {$\textcolor{red}{\mc{Z}_0^r}$};
	\node at (1.3,1.8) {$\textcolor{blue}{\mc{Z}_0^a}$};
	\end{tikzpicture}
		\caption{$\lambda = 1$\figlab{selfint_b}}
	\end{subfigure}
	\caption{(a) Transcritical scenario from \cite{krupa2001extending}; (b) delayed loss of stability along an invariant portion of $\mc{M}_0$.}
	\figlab{selfint}
\end{figure}

It is therefore evident that the reformulation of \eqref{genmat} in polar coordinates, as given by \eqref{pc0}, is useful for identifying the eigendirection along which trajectories in a $\delta$-neighbourhood of $\mc{C}_0$ accumulate contraction or expansion, for various values of $x$. We have shown that, depending on the properties of the auxiliary system in \eqref{pc0}, we can distinguish between two cases: in the first case, trajectories of \eqref{genmat} switch the eigendirection they follow as soon as this eigendirection becomes repelling, as seen in \figref{selfint_a}; in the second case, trajectories exhibit entry-exit behaviour along the eigendirection they were initially attracted to, before being attracted to the other eigendirection, as shown in \figref{selfint_b}. This distinction indicates that the corresponding formulae for the accumulated contraction and expansion are given by \eqref{trans} and \eqref{invar}, respectively. Finally, we remark that, in general, due to the rotation of the linear subspaces of \eqref{zsub} along $\mc{C}_0$, as indicated by the given expressions for $\theta_{i}(x)$ ($i=1,2$), trajectories that enter the $\delta$-neighbourhood of $\mc{C}_0$ at some point with $z_1>0$ could potentially exit at some point with $z_1<0$; recall \figref{3dsf}. (A similar statement applies to the signs of $z_2$.)

\section{Examples}
\seclab{sys}
In this section, we present a number of examples that illustrate our main result, \thmref{thrm}.

\subsection{A one-way coupled system}
As our first example, we consider the system
\begin{subequations}\eqlab{syst1}
\begin{align}
    x'&=\varepsilon,\eqlab{syst1-a}\\
    z_1'&= x z_1,\eqlab{syst1-b}\\
    z_2'&= x z_1-z_2, \eqlab{syst1-c}
\end{align}
\end{subequations}
where we observe that the variables $ (x, z_1)$ are decoupled from $ z_2 $. The corresponding critical manifold for \eqref{syst1} is given by $\mc{C}_0=\{z_1=0=z_2\}$, {where the eigenvalues of the linearisation of the fast $(z_1,z_2)$-subsystem in \eqref{syst1} along $ \mc{C}_0 $ read} 
\begin{equation}\eqlab{ex1eigen}
\mu_1(x)=-1\quad  \tn{and} \quad \mu_2(x)= x.
\end{equation}
At $ x^* = -1 $, it holds that
\begin{align*}
\mu_1(x^*)=-1=\mu_2(x^*).
\end{align*} 
For $ x>-1 $, equation~\eqref{syst1} is diagonalisable, with one eigendirection that changes stability from stable to unstable, and an eigendirection that is always stable. The corresponding eigenvalues in these directions are $ \mu_1(x) $ and $ \mu_2(x) $, respectively. Therefore, standard theory on delayed loss of stability can be employed by considering only the eigendirection along which the stability changes \cite{de2008smoothness,de2016entry}, recall \eqref{wio-pl1}, and the entry-exit function is of the form
\begin{align*}
\int_{x_0}^{x_1} x \text{d}x = 0, \quad \text{which implies} \quad x_1 = - x_0
\end{align*}
for $ x_0\in(-1,0) $. In the following, we consider the case where $ x_0<-1 $. 

After transformation to polar coordinates, \eqref{syst1} reads
\begin{subequations}\eqlab{polar1}
	\begin{align}
	x' &= \varepsilon,\\
	\theta' &= x\cos^2\theta-(x+1)\sin\theta\cos\theta=:\Phi(x,\theta, \varepsilon)
	\end{align}
\end{subequations}
for $r=0$; cf.~\eqref{pc0}.
When $\varepsilon=0$, the critical manifold $ \mc{M}_0 $ for \eqref{polar1} is given by \eqref{crima}; in particular, it consists of two branches in this case. The first branch $\mc{S}_0$ is obtained from
\begin{equation}\eqlab{ex1S0}
\cos \theta = 0, \quad\text{which implies}\quad
\theta_1(x) = -\frac{\pi}{2},
\end{equation}
whereas the second branch $\mc{Z}_0$ is defined implicitly by
\[
x\cos\theta-(x+1)\sin\theta=0.
\]
These branches intersect at $ (x,\theta)=(-1,-\frac{\pi}{2}) $; see \figref{figure2} for an illustration. We note that the branch $\mc{Z}_0$ can be represented as
\begin{align}\eqlab{ex1Z0}
\theta_2(x) =
\begin{cases}
 \arctan\lp \frac{x}{x+1}\rp & x \neq -1,\\
 -\dfrac{\pi}{2} & x=-1,
\end{cases}
\end{align}
where the extension with continuity is a consequence of our identification of the angular variable modulo $\pi$. The explicit representation of $\theta$ as a function of $x$ naturally breaks down at $x=-1$ due to the fact that \eqref{polar1} undergoes a transcritical bifurcation at $(x,\theta)=(-1,-\frac{\pi}{2})$.

We emphasise that the branch $\mc{S}_0$ of $ \mc{M}_0 $ is invariant for equation~\eqref{polar1} with $ \varepsilon>0 $, and we reiterate that the angle $ \theta_1(x) $ corresponds to the eigenvector $(1,0)$ in $(z_1,z_2)$-coordinates, associated with the eigenvalue $\mu_1(x)=-1$, whereas the angle $ \theta_2(x) $ corresponds to the eigenvector $(x,x+1)$, associated with the eigenvalue $\mu_2(x)=x$. 

From \eqref{phix}, we obtain
\begin{equation}\eqlab{ex1deriv}
\dfrac{\partial \Phi}{\partial \theta}(x,\theta,0)=-(x+1)(\cos^2\theta-\sin^2\theta) -2x\cos\theta \sin \theta,
\end{equation}
which implies that the branch $\mc{S}_0$ is attracting for $x<-1$ and repelling if $x>-1$; correspondingly, $\mc{Z}_0$ is repelling for $x<-1$ and attracting when $x>-1$, as illustrated again in \figref{figure2}.

\begin{figure}[h!]\centering
	\begin{tikzpicture}
	\node at (0,0){
		\includegraphics[width=0.4\textwidth]{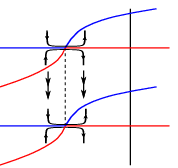}};
	\node at (-4,1.5) {$\theta = \frac{\pi}{2}$};
	\node at (-4,-1.5) {$\theta = -\frac{\pi}{2}$};
	\node at (-0.8,0.2) {\small$x=-1$};
	\node at (-1.75,0.9) {\small$x_0$};
	\node at (0.15,0.9) {\small$\tilde{x}$};
	\node at (-3,1.9) {$\textcolor{blue}{\mc{S}_0^a}$};
	\node at (1.3,1.2) {$\textcolor{red}{\mc{S}_0^r}$};
	\node at (-3,0.5) {$\textcolor{red}{\mc{Z}_0^r}$};
	\node at (1.3,3.2) {$\textcolor{blue}{\mc{Z}_0^a}$};
	
	\node at (-3,-1.15) {$\textcolor{blue}{\mc{S}_0^a}$};
	\node at (1.3,-1.8) {$\textcolor{red}{\mc{S}_0^r}$};
	\node at (-3,-2.5) {$\textcolor{red}{\mc{Z}_0^r}$};
	\node at (1.3,0.1) {$\textcolor{blue}{\mc{Z}_0^a}$};
	\end{tikzpicture}
	\caption{Stability of the branches $ \mc{S}_0 $ and $ \mc{Z}_0 $ of the critical manifold $ \mc{M}_0 $ of \eqref{polar1} (blue: attracting; red: repelling). The delayed loss of stability of the singularity at $(-1,\pm\frac{\pi}{2})$ can be studied via the classical entry-exit formula. Note that the horizontal lines $\theta=\pm\frac{\pi}{2}$ are invariant for $\varepsilon>0$, and that they are naturally identified due to the definition of $\mc{M}_0$ in \eqref{crima}.}
	\figlab{figure2}
\end{figure}

For the parameters defined in \thmref{thrm}, we calculate
\begin{equation*}
\begin{split}
\alpha &= \frac{1}{2}\dfrac{\partial^2 \Phi}{\partial \theta^2}(-1,-\tfrac{\pi}{2},0)= -1,\quad
\beta = \frac{1}{2}\dfrac{\partial^2 \Phi}{\partial x \partial \theta}(-1,-\tfrac{\pi}{2},0)= \frac{1}{2},\\
\gamma &= \frac{1}{2}\dfrac{\partial^2 \Phi}{\partial x^2}(-1,-\tfrac{\pi}{2},0)= 0,\quad\text{and}\quad
\delta = \frac{1}{2}\dfrac{\partial \Phi}{\partial \varepsilon}(-1,-\tfrac{\pi}{2},0)=0,
\end{split}
\end{equation*}
which implies that
\begin{equation*}
\lambda = \dfrac{\delta\alpha +\beta}{\sqrt{\beta^2-\gamma\alpha}}=1; 
\end{equation*}
the corresponding entry-exit function is therefore given by \eqref{invar}, i.e., by
\begin{align}
\int_{x_0}^{\tilde{x}} (-1) \text{d}x + \int_{\tilde{x}}^{x_1} x\text{d}x =0, \quad \text{which implies} \quad x_0 - \tilde{x} +\frac{x_1^2}{2}- \frac{\tilde{x}^2}{2}=0; \eqlab{wio1}
\end{align}
here, the auxiliary coordinate $\tilde{x}$ is calculated via \eqref{xtil} as
\begin{equation}\eqlab{relaz1}
\int_{x_0}^{\tilde{x}} (x+1) \text{d}x=0, \quad \text{which implies} \quad x_0+ \frac{x_0^2}{2} = \tilde{x}+\frac{\tilde{x}^2}{2}.
\end{equation}
Combining \eqref{wio1} and \eqref{relaz1}, we conclude
\begin{equation}\eqlab{exit1}
x_1=-x_0.
\end{equation}

Hence, for $\varepsilon>0$ sufficiently small, we observe a typical delayed loss of stability in \eqref{polar1} after $x=-1$, with the attracting branch $ \theta = \theta_1(x) $ becoming repelling. 
It follows that, for $x\in (x_0,\tilde{x})$, trajectories of \eqref{polar1} ``choose'' $\theta = \theta_1(x)$, i.e., the eigenvalue $\mu_1(x)=-1$ with corresponding eigendirection $(0,1)$ in \eqref{syst1}, while for $x>\tilde{x}$, trajectories ``choose'' $\theta = \theta_2(x)$, i.e., the eigenvalue $\mu_2(x)=x$ with eigendirection $(x,x+1)$. We remark that, in \eqref{syst1}, the small parameter $\varepsilon$ only changes the speed at which the slow variable $x$ evolves; hence, orbits in the $(x,z_1)$- and $(x,z_2)$-planes are identical, up to a rescaling of the time variable, for different positive values of $\varepsilon$. In particular, the exit point in \eqref{exit1} is independent of $ \varepsilon $ as long as $\varepsilon$ is sufficiently small.

\subsection{Coupled systems with $\varepsilon$--dependence}
Next, we consider the system
\begin{subequations}\eqlab{syst2}
\begin{align}
    x'&=\varepsilon,\\
    z_1'&= x z_1-\varepsilon z_2,\\
    z_2'&= x z_1-z_2
\end{align}
\end{subequations}
where, in contrast to \eqref{syst1}, the variables $ (z_1, z_2) $ are now two-way coupled for $\varepsilon >0$. The corresponding critical manifold for \eqref{syst2} again reads $\mc{C}_0=\{z_1=0=z_2\}$, {where the eigenvalues of the linearisation of the fast $(z_1,z_2)$-subsystem in \eqref{syst2} along $ \mc{C}_0 $ are given by} \eqref{ex1eigen},
as before; hence, we again have
$\mu_1(x^*)=-1=\mu_2(x^*)$ at $ x^* = -1 $. In polar coordinates, \eqref{syst2} becomes
\begin{subequations}\eqlab{polar2}
\begin{align}
x' &= \varepsilon,\\
\theta' &= x\cos^2\theta-(x+1)\sin\theta\cos\theta+\varepsilon  \sin^2\theta=:\Phi(x,\theta,\varepsilon).
\end{align}
\end{subequations}
The critical manifold $ \mc{M}_0 $ of \eqref{polar2} is again defined as in \eqref{crima} and consists of two branches given by \eqref{ex1S0} and \eqref{ex1Z0}, as before, which
which intersect at $ x=-1 $; see \figref{figure3} for an illustration. We emphasise that the branches \eqref{ex1S0} and \eqref{ex1Z0} of $ \mc{M}_0 $ are not invariant for \eqref{polar2} with $ \varepsilon>0 $. From \eqref{phix}, we again obtain \eqref{ex1deriv}, as before,
which again implies that $\mc{S}_0$ is attracting for $x<-1$ and repelling if $x>-1$, whereas $\mc{Z}_0$ is repelling for $x<-1$ and attracting when $x>-1$.

\begin{figure}[h!]\centering
	\begin{tikzpicture}
	\node at (0,0){
		\includegraphics[width=0.4\textwidth]{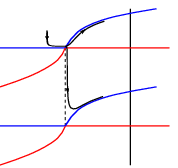}};
	\node at (-4,1.5) {$\theta = \frac{\pi}{2}$};
	\node at (-4,-1.5) {$\theta = -\frac{\pi}{2}$};
	\node at (-1.4,0) {\small$x=-1$};
	\node at (-1.5,2.4) {\small$x_0$};
	\node at (-3,1.9) {$\textcolor{blue}{\mc{S}_0^a}$};
	\node at (1.3,1.2) {$\textcolor{red}{\mc{S}_0^r}$};
	\node at (-3,0.5) {$\textcolor{red}{\mc{Z}_0^r}$};
	\node at (1.3,3.2) {$\textcolor{blue}{\mc{Z}_0^a}$};
	
	\node at (-3,-1.15) {$\textcolor{blue}{\mc{S}_0^a}$};
	\node at (1.3,-1.8) {$\textcolor{red}{\mc{S}_0^r}$};
	\node at (-3,-2.5) {$\textcolor{red}{\mc{Z}_0^r}$};
	\node at (1.3,0.1) {$\textcolor{blue}{\mc{Z}_0^a}$};
	
	\end{tikzpicture}
	\caption{Stability of the branches of the critical manifold of system \eqref{polar2}. Blue: stable; red: unstable. Notice that the horizontal lines $\theta=-\frac{\pi}{2},\frac{\pi}{2}$, corresponding to eigenvalue $\mu_2(x)=-1$ are \emph{not} invariant for $\varepsilon>0$. The curves correspond to eigenvalue $\mu_1(x)=x$, which eventually causes the loss of stability.}
	\figlab{figure3}
\end{figure}

For the parameters defined in \thmref{thrm}, we calculate
\begin{equation*}
\begin{split}
\alpha &= \frac{1}{2}\dfrac{\partial^2 \Phi}{\partial \theta^2}(-1,-\tfrac{\pi}{2},0)= -1,\quad
\beta = \frac{1}{2}\dfrac{\partial^2 \Phi}{\partial x \partial \theta}(-1,-\tfrac{\pi}{2},0)= \frac{1}{2},\\
\gamma &= \frac{1}{2}\dfrac{\partial^2 \Phi}{\partial x^2}(-1,-\tfrac{\pi}{2},0)= 0,\quad\text{and}\quad
\delta = \frac{1}{2}\dfrac{\partial \Phi}{\partial \varepsilon}(-1,-\tfrac{\pi}{2},0)= \frac{1}{2},
\end{split}
\end{equation*}
which implies that
\begin{equation*}
\lambda = \dfrac{\delta\alpha +\beta}{\sqrt{\beta^2-\gamma\alpha}}=0; 
\end{equation*}
the entry-exit function is therefore given by \eqref{trans}, i.e., by
\begin{align}
\int_{x_0}^{-1}(-1)\text{d}x + \int_{-1}^{x_1} x\text{d}x =0, \eqlab{wio2}
\end{align}
from which we calculate
\begin{equation}\eqlab{exit2}
    x_1=\sqrt{-1-2x_0}.
\end{equation}

We emphasise that, although equations~\eqref{syst1} and \eqref{syst2} differ in the $ \mc{O}(\varepsilon) $-terms of the $ z_1 $-equation only, the corresponding exit points obtained via \eqref{exit1} and \eqref{exit2} are different. We further reiterate that the formula in \eqref{wio2} is valid for initial conditions with $ x_0<-1 $, as for $ x \in [-1,0)$, \eqref{syst2} is diagonalisable with one direction that is always attracting; therefore, delayed loss of stability can be studied solely in the other direction along which the stability changes from attracting to repelling. In \figref{exit2}, we compare our prediction for the exit point -- which is given by \eqref{exit2} for $x_0<-1$ and by $x_1=-x_0$ for $x_0 \in [-1,0)$ -- with a numerical integration of \eqref{syst2}, where $\varepsilon=0.01$. The resulting figure highlights a very close match between the two curves. Moreover, in \figref{fig:timeser}, we illustrate orbits of \eqref{syst2} for varying values of $\varepsilon$, as indicated in the legend; the initial condition is set to $(x,z_1,z_2)(0)=(-2,1,1)$ throughout. We note that an increase in $\varepsilon$ decreases the exit time and ``smoothens'' orbits near $z_1=0$.
\begin{figure}[h!]\centering
	\includegraphics[width=0.6\textwidth]{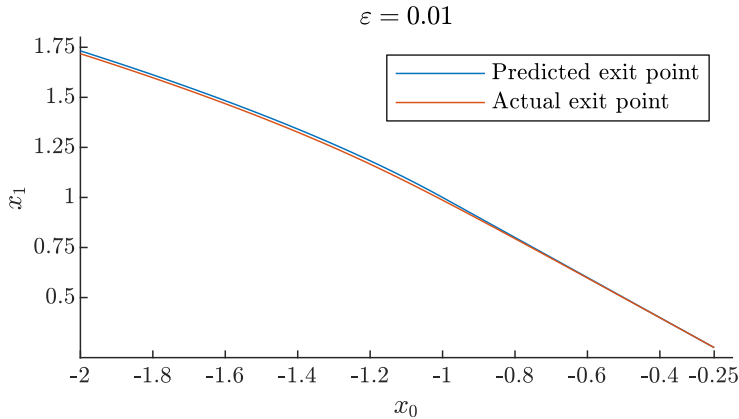}
	\caption{Exit points for entry points in the interval $x_0\in (-2,-0.25)$, as predicted by the formulae in \eqref{exit2} and \eqref{exit1} (blue) and as obtained by direct integration of equation~\eqref{syst2}, with $\varepsilon=0.01$ (red). The initial values for the fast variables are chosen as $(z_1,z_2)(0,0)=(1,1)$. Note that, after $x=-1$, the exit point coincides with the prediction from the standard entry-exit formula, which implies $x_1=-x_0$. The error is consistent with the expected $\mathcal{O}(\varepsilon)$-distance between the two curves.}
	\figlab{exit2}
\end{figure}
\begin{figure}[h!]\centering
	\includegraphics[width=0.6\textwidth]{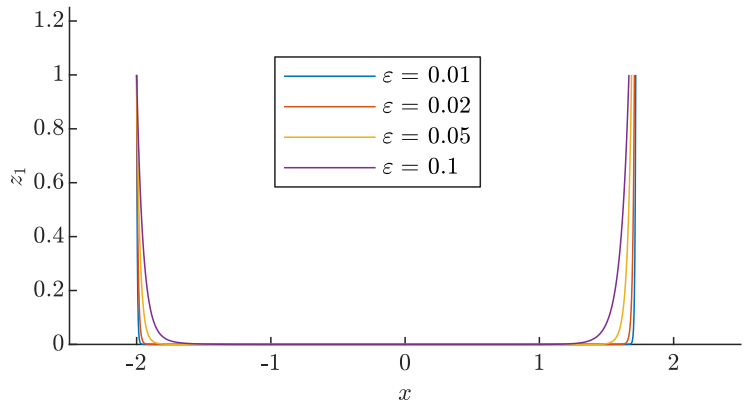}
	\caption{Exit points for equation~\eqref{syst2} with initial condition at $(x,z_1,z_2)(0)=(-2,1,1)$, where $\varepsilon$ varies as indicated. The formula in \eqref{exit2} predicts $x_1=\sqrt{3}\approx 1.732$ in this example; the actual exit point for $\varepsilon=0.01$ is found at $x\approx 1.718$.}
	\figlab{fig:timeser}
\end{figure}

\bigskip

Finally, in order to demonstrate that higher-order terms in $z_i$ ($i=1,2$) in the $ (z_1, z_2) $-subsystem do not locally affect the delay phenomena studied here, we consider the system
\begin{subequations}\eqlab{syst4}
	\begin{align}
        x'&=\varepsilon,\\
        z_1'&= x \bigg(z_1-\frac{z_1^2}{a}\bigg)+\varepsilon z_2,\\
        z_2'&= z_1^2-z_2
\end{align}
\end{subequations}
for some $a>0$.
The corresponding critical manifold is now given by $\mathcal{C}_0:=\{ z_1 = 0 = z_2 \} \cup \{ z_1=a \ \tn{and}\ z_2=a^2 \}$. Choosing $a>0$ sufficiently large, we can focus on the first portion of $\mc{C}_0$ and study the corresponding entry-exit function without considering the second, $a$-dependent portion. The eigenvalues of the linearisation of the $ (z_1, z_2) $-subsystem in \eqref{syst4} about $\lb z_1 = 0 = z_2\rb$ for $\varepsilon=0$ are again given by \eqref{ex1eigen}. Transformation to polar coordinates yields
\begin{equation*}
\begin{split}
    x' &= \varepsilon,\\
    \theta' &= -(1+x)\sin\theta\cos\theta+\varepsilon  \sin^2\theta=:\Phi(x,\theta,\varepsilon),
\end{split}
\end{equation*}
where we consider $r=0$ only, as above.

For the parameters in \thmref{thrm}, we calculate $\alpha=0$, $\beta=-\frac{1}{2}$, $\gamma=0$, and $\delta=\frac{1}{2}$ which, by \eqref{lam}, implies $\lambda = -1 \neq 1$; therefore, the entry-exit function is given by \eqref{trans}. The entry-exit formula in this example is hence again defined by \eqref{wio2}, with the $ x$-coordinate of the exit point explicitly given by \eqref{exit2}. Indeed, for an entry point with $x_0<-1$, the corresponding exit point satisfies $x_1=\sqrt{-1-2x_0}$, as seen in \figref{fig:exit4}, where we have chosen $a=4$ and $(z_1,z_2)(0)=(0.5,0.5)$.

\begin{figure}[h!]\centering
	\includegraphics[width=0.6\textwidth]{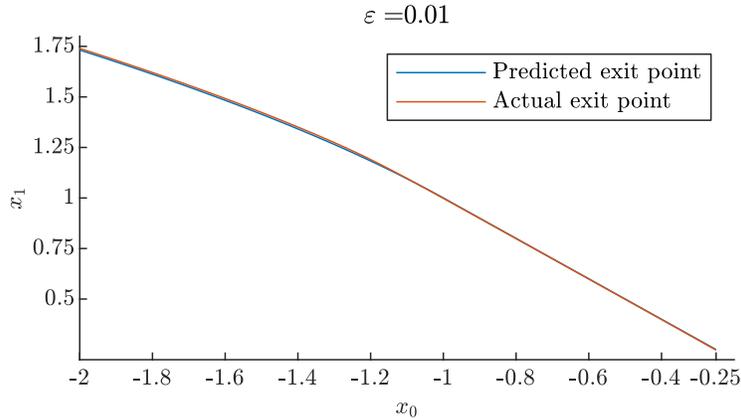}
	\caption{Exit points for entry points in the interval $x_0\in (-2,-0.25)$, as predicted by the formulae in \eqref{exit2} and \eqref{exit1} (blue) and as obtained by direct integration of \eqref{syst4}, with $a=4$ fixed and $\varepsilon=0.01$ (red). The initial values for the fast variables are chosen as $(z_1,z_2)(0)=(0.5,0.5)$. Note that, after $x=-1$, the exit point coincides with the prediction from the standard entry-exit formula, which implies $x_1=-x_0$. The error is consistent with the expected $\mathcal{O}(\varepsilon)$-distance between the two curves.}
	\figlab{fig:exit4}
\end{figure}

\section{Conclusions and outlook}\seclab{conc}
In this paper, we have studied the phenomenon of delayed loss of stability along one-dimensional critical manifolds in fast-slow systems with two fast and one slow variables, where the linearisation of the corresponding fast subsystem about that manifold has two real eigenvalues. More precisely, we have focused on the scenario where two of these eigenvalues coincide for some value of the slow variable before at least one of them becomes positive; hence, the ``leading''  eigenvalue and the corresponding ``stronger'' eigendirection change along the critical manifold, which renders the use of previously known entry-exit formulae unsuitable. 

Via a transformation to polar coordinates, we have uncovered the hidden structure of these systems, and we have proposed a methodology for deriving extended entry-exit formulae which cover different qualitative scenarios. We have illustrated our findings for a few simple prototypical examples, and we have verified them by numerical simulation. Notably, our analysis shows that a leading-order linearisation of the vector field about the corresponding critical manifold is sufficient for constructing entry-exit formulae in a robust fashion and for estimating accurately the resulting exit points after a delayed loss of stability. 

The phenomenon of ``crossing'' eigenvalues studied here is ubiquitous in systems with more than two fast variables. It may potentially occur also for more than one slow variable as in variants of  the models studied in \cite{boudjellaba2009dynamic,deng2003food,desroches2012mixed,jardon2020fast,jardon2020geometric,kuehn2015multiscale}. We postulate that our construction can be extended to such systems, by careful consideration of each intersection of the eigenvalues along the corresponding critical manifold. We leave a potential classification and extension of entry-exit formulae in analogy to those derived in \thmref{thrm} in higher dimensions, as well as the investigation of alternatives to the polar coordinate transformation performed here, for future work.

\section*{Acknowledgements}

The authors would like to thank Stephen Schecter for insightful discussions and relevant recommendations. 
CK thanks the VolkswagenStiftung for support via a Lichtenberg Professorship. CK also thanks the DFG for support via a Sachbeihilfe Grant.

\bibliographystyle{plain}
\bibliography{biblio}

\end{document}